\documentclass[reqno,a4paper]{amsart}%%{ctexart}[a4paper]
\usepackage{amsmath,amstext,amssymb,amsfonts,amscd,amsthm}
\usepackage{geometry}
\usepackage{esint}
\usepackage{hyperref}
\hypersetup{pdfpagemode=FullScreen,colorlinks=true,linkcolor=blue,citecolor=blue,urlcolor=orange}
\usepackage{mathdots,mathrsfs,enumerate}
\usepackage{subfigure}
\usepackage{extarrows}
\usepackage{mathrsfs}
\usepackage{dsfont}
\usepackage{graphicx,color}
\usepackage{tikz}
\usetikzlibrary{3d,calc,patterns,graphs,arrows}
\usepackage{tikz-cd,array,diagbox}

\numberwithin{equation}{section}
\newenvironment{proof*}{\noindent{\heiti 证明}}{\hfill\qed}%
\newtheorem{definition}{Definition}[section]

\newtheorem{lemma}{Lemma}[section]
\newtheorem{theorem}{Theorem}[section]
\newtheorem{corollary}{Corollary}[section]
\newtheorem{proposition}{Proposition}[section]

\newtheorem{remark}{Remark}[section]

\newcommand{\A}{\forall}

\newcommand{\p}{\partial}

\newcommand{\ls}{\lesssim}

\newcommand{\norm}[1]{\left\Vert#1\right\Vert}
\newcommand{\set}[1]{\left\{#1\right\}}

\newcommand{\frb}[1]{\left(#1\right)}

\newcommand{\wt}{\widetilde}

\newcommand{\Z}{\mathbb Z}

\newcommand{\R}{\mathbb R}

\renewcommand{\a}{\alpha}

\newcommand{\g}{\gamma}

\renewcommand{\d}{\delta}

\newcommand{\ve}{\varepsilon}
\renewcommand{\th}{\theta}

\renewcommand{\l}{\lambda}

\newcommand{\s}{\sigma}

\newcommand{\vp}{\varphi}

\newcommand{\Om}{\Omega}

\allowdisplaybreaks[3]

\begin{document}

\title{Regularity criteria for the surface growth model with a forcing term}

\author{Yuqian Cheng}
\address{
School of Mathematics and Center for Nonlinear Studies,
Northwest University, Xi'an, 710127, PR China}
\email{chengyuqian@stumail.nwu.edu.cn}

\author{Zhisu Li}
\address{
School of Mathematics and Center for Nonlinear Studies,
Northwest University, Xi'an, 710127, PR China}
\email{lizhisu@nwu.edu.cn}

\author{Xuening Wei}
\address{
School of Mathematics and Center for Nonlinear Studies,
Northwest University, Xi'an, 710127, PR China}
\email{xnwei.math@qq.com}

\date{March 13, 2026}

\maketitle

\begin{abstract}
Based on a compactness method,
we establish regularity criteria for suitable weak solutions
to the surface growth model with a forcing term.
These criteria imply that the H\"older regularity of solutions
follows from smallness conditions on several scale-invariant quantities.
As a consequence, we obtain a partial regularity result stating that
the one-dimensional biparabolic Hausdorff measure of the singular set is zero.
\end{abstract}

\textbf{Key words.}
surface growth model, suitable weak solutions,
partial regularity, compactness method

\textbf{AMS subject classifications.}
35B65, % Smoothness and regularity of solutions to PDEs
35K55, % Nonlinear parabolic equations
35Q35 % PDEs in connection with fluid mechanics

%\tableofcontents

\section{Introduction}

In this paper,
we study the one-dimensional surface growth model with a forcing term
\begin{equation}\label{eqn.SGM-f}
u_t+u_{xxxx}+\p_{xx}u_x^2=f \quad \text{in}\ Q_1,
\end{equation}
where $u$ represents the height of a crystalline layer,
and $f$ denotes a given external force.

In the absence of forcing, \eqref{eqn.SGM-f} reduces to
\begin{equation}\label{eqn.SGM}
u_t+u_{xxxx}+\p_{xx}u_x^2=0.
\end{equation}
This model arises in the study of epitaxial growth of monocrystals
and was previously investigated by Bl\"omker and Romito \cite{BR09},
who established several regularity results and blow-up criteria.
In a subsequent work \cite{BR12},
they proved the local existence and uniqueness
of weak solutions in the largest critical space.
As observed by Bl\"omker and Romito,
these results reveal that \eqref{eqn.SGM}
shares several striking similarities with the Navier--Stokes equations.
Such analogies have motivated the development of
a partial regularity theory for \eqref{eqn.SGM},
drawing on the extensive literature on the Navier--Stokes equations.

A seminal result in this direction for the Navier--Stokes equations
was established by Caffarelli, Kohn, and Nirenberg \cite{CKN82},
who proved that the one-dimensional parabolic Hausdorff measure of
the singular set of suitable weak solutions is zero.
More than a decade later,
their proof was simplified by Lin \cite{L98} and Ladyzhenskaya and Seregin \cite{LS99},
using a rescaling approach combined with blow-up arguments.
Vasseur \cite{V07} subsequently provided an alternative proof
by employing De Giorgi's iteration.
More recently, Wang \cite{W25} introduced a compactness method based on
a compactness lemma and the monotonicity property of harmonic functions.
Related results can also be found in \cite{S77,S80,S88,TX99} and the references therein.

Adapting the rescaling approach,
O\.za\'nski and Robinson \cite{OR19} studied \eqref{eqn.SGM}
on the space-time domain $\mathbb T\times(0,\infty)$,
where $\mathbb T$ denotes the one-dimensional torus.
In particular, they established the following $\ve$-regularity criteria:
there exist absolute constants $\ve_0,\ve_1>0$ such that,
if $u$ is a suitable weak solution of \eqref{eqn.SGM} and
\begin{equation}\label{eqn.criteria-1}
\frac{1}{r^2}\int_{Q_r}|u_x|^3\leq\ve_0,
\end{equation}
then $u$ is H\"older continuous in $Q_{r/2}$.
Moreover, if either
\begin{equation}\label{eqn.criteria-2}
\limsup_{r\to0}\frac{1}{r}\int_{Q_r}u_{xx}^2\leq\ve_1
\quad \text{or} \quad
\limsup_{r\to0}\sup_{-r^4<s<r^4}\frac{1}{r}\int_{B_r}|u(\cdot,s)|^2\leq\ve_1,
\end{equation}
then $u$ is H\"older continuous in $Q_\rho$ for some $\rho>0$.
These criteria, in turn, imply the partial regularity results:
the box-counting dimension of the singular set is at most $7/6$,
and its one-dimensional biparabolic Hausdorff measure is zero.
Here, the singular set is defined as the collection of points
where the solution is not H\"older continuous.

Subsequent studies have further refined this regularity theory.
Choi and Yang \cite{CY19} proposed a new regularity criterion
in terms of mixed Lebesgue norms.
O\.za\'nski \cite{O19} derived a Serrin-type condition
for weak solutions of \eqref{eqn.SGM},
namely, $u_x\in L^{p}(0,1;L^q(B_1))$, $4/p+1/q\leq1$, $p,q>1$,
under which $u$ is locally $C^\infty$ smooth.
Later, Burczak, O\.za\'nski, and Seregin \cite{BOS20}
proved the criterion \eqref{eqn.criteria-1} also implies the H\"older continuity of $u_x$,
thereby ensuring the smoothness of $u$.
Relying on this result,
they showed that the one-dimensional biparabolic Hausdorff measure of the set,
where the solution fails to be smooth, vanishes.
In addition,
Wang, Huang, Wu, and Zhou \cite{WHWG23} studied a model more general than \eqref{eqn.SGM-f},
establishing the corresponding regularity criteria
and obtaining partial regularity results via the rescaling approach.
For additional results, we refer the reader to \cite{Y21,WWH23,WWLY24}.

Building on these developments,
we establish regularity criteria for the forced model \eqref{eqn.SGM-f} that
are analogous to \eqref{eqn.criteria-1} and \eqref{eqn.criteria-2}
by employing the compactness method developed in \cite{W25}.
This method relies on a linear approximation argument,
which offers a clear and transparent framework for the proof.
In particular,
the results show that the H\"older regularity of $u$ is ensured by smallness conditions
on several scale-invariant quantities.

As a further direction,
using the same method,
we have established the H\"older regularity criterion for $u_x$ of the forced model \eqref{eqn.SGM-f},
see \cite{CL26}.
Moreover, we also consider the $n$-dimensional surface growth model of $\g$ type 
and the minimal deposition equation of an amorphous thin film growth model,
and have established the corresponding regularity criteria
and the partial regularity results, see \cite{CLW26-1,CLW26-2}.

Before stating the main results,
we introduce several scale-invariant quantities
that will be used to formulate the regularity criteria.
These quantities are motivated by the natural scaling invariance of \eqref{eqn.SGM-f},
namely,
\[
u^r(x,t):=u(rx,r^4t),
\quad
f^r(x,t):=r^4f(rx,r^4t),\ \, r>0.
\]
This scaling plays a central role in the regularity analysis of solutions,
as it relates their behavior across different scales.
Using this scaling structure, we define
\begin{align*}
G_r[u]&:=r\frb{\fint_{Q_r}|u_x|^3}^{1/3},
\\
U_r[u]&:=\sup_{-r^4<s<r^4}\frac{1}{r}\int_{B_r}|u(\cdot,s)|^2,
\\
O_r[u]&:=\sup_{-r^4<s<r^4}\frac{1}{r}\int_{B_r}|u(\cdot,s)-u_{B_r}(s)|^2,
\\
L_r[u]&:=\frac{1}{r}\int_{Q_r}u_{xx}^2,
\\
F_r[f]&:=r^4\frb{\fint_{Q_r}|f|^p}^{1/p}.
\end{align*}
More precisely, the scaling invariance implies
\[
G_1[u^r]=G_r[u],\quad U_1[u^r]=U_r[u],\quad O_1[u^r]=O_r[u],
\quad L_1[u^r]=L_r[u],\quad F_1[f^r]=F_r[f].
\]
When the dependence on $u$ and $f$ is clear,
we abbreviate the notation by writing $G_r$, $U_r$, $O_r$, $L_r$ and $F_r$.
Note that $O_r\leq U_r$ (by Lemma \ref{lem.mean-int}),
and $G_r$ satisfies the translation-invariance property
\begin{equation}\label{eqn.invarp}
G_r[u-a]=G_r[u]\quad \text{for any}\ a\in\R,
\end{equation}
which will be used in the approximation argument in Lemma~\ref{lem.appro}.

We now state the main theorems.
Let $u$ be a suitable weak solution of \eqref{eqn.SGM-f}
with $f \in L^p(Q_1)$, where $p>3/2$.
The following results hold:

\begin{theorem}\label{thm.main-G+F}
There exist universal constants $\d_0=\d_0(p)>0$ and $0<\a=\a(p)<1$
such that for any $0<r\leq1$, if
\[
G_r[u]+F_r[f]^{1/2}\leq\frac{\d_0}{2},
\]
then $u\in C^{\a,\a/4}(Q_{r/4})$.
\end{theorem}

\begin{theorem}\label{thm.main-U}
There exist universal constants $\d_1=\d_1(p)>0$ and $0<\a=\a(p)<1$
such that for any $0<\rho\leq1$, if
\[
\sup_{0<r\leq\rho}U_r[u]\leq\d_1,
\]
then $u\in C^{\a,\a/4}(Q_{\rho r_0/4})$ for some $r_0>0$.
\end{theorem}

\begin{theorem}\label{thm.main-L}
There exist universal constants $\d_2=\d_2(p)>0$ and $0<\a=\a(p)<1$
such that for any $0<\rho\leq1$, if
\[
\sup_{0<r\leq\rho}L_r[u]\leq\d_2,
\]
then $u\in C^{\a,\a/4}(Q_{\rho r_0/4})$ for some $r_0>0$.
\end{theorem}

Without loss of generality,
we focus our analysis on the origin.
Throughout this paper,
we will not repeat the assumption that $p>3/2$.
The notion of suitable weak solutions can be found in Definition~\ref{def.sws}.

\begin{remark}
Theorem \ref{thm.main-G+F} and Theorem \ref{thm.main-U}
are established on the space-time cylinder $Q_1\subset\R\times\R$,
whereas Theorem \ref{thm.main-L} is proved
in the periodic spatial setting $Q_1\subset\mathbb T\times\R$.
This distinction arises from
the spatial interpolation inequality \eqref{eqn.G-O-L},
which plays a crucial role in the proof of Theorem \ref{thm.main-L}
and holds only on the one-dimensional torus $\mathbb T$.
\end{remark}

\begin{remark}
The compactness method introduced by Wang \cite{W25}
relies on three key lemmas:
the approximation lemma (Lemma \ref{lem.appro}),
the contraction lemma (Lemma \ref{lem.Contra}),
and the decay lemma (Lemma \ref{lem.Iterative}).
The approximation lemma is inspired by Wang's compactness lemma (\cite[Lemma 6]{W25})
but is adapted to our setting,
and it serves as a crucial step in establishing the contraction lemma.
Furthermore, through an iterative argument,
the contraction lemma leads to the decay lemma.
Together, these lemmas provide a rigorous framework
for deriving the desired regularity criteria.
\end{remark}

\begin{remark}
The regularity criteria \eqref{eqn.criteria-1} and \eqref{eqn.criteria-2}
in \cite[Theorem 1.1]{OR19} yield H\"older continuity of
suitable weak solutions to \eqref{eqn.SGM} for every exponent $0<\a<1$.
In contrast, our main results establish
the existence of some $0<\a<1$ for \eqref{eqn.SGM-f},
which is intrinsic to the method employed
rather than a consequence of the presence of the forcing term.
\end{remark}

Using the above regularity criteria,
we proceed to estimate the biparabolic Hausdorff dimension (Definition \ref{def.phm})
of the singular set of a suitable weak solution $u$ to \eqref{eqn.SGM-f}.
Here, the singular set is defined by
\[
S_u^\a:=\set{(x,t)\in Q_1\subset\mathbb T\times\R: u\notin C^{\a,\a/4}(Q_r(x,t)),\
\A Q_r(x,t)\subset Q_1},\ \  0<\a<1,
\]
where the biparabolic cylinder $Q_r(x,t)$ is given by \eqref{def.Q-r}.

\begin{corollary}[Partial regularity]\label{cor.Partial}
Let $u$ be a suitable weak solution of \eqref{eqn.SGM-f}
with $f\in L^p(Q_1)$, where $p>3/2$.
Then there exists a universal constant $0<\a=\a(p)<1$ such that
\[
\mathcal P^1(S^\a_u)=0,
\]
where $\mathcal P^1$ denotes the 1-dimensional biparabolic Hausdorff measure.
\end{corollary}

We remark that Corollary \ref{cor.Partial} follows from a standard covering argument,
which combines Theorem \ref{thm.main-L} with the Vitali covering lemma.

The article is organized as follows.
In Section~\ref{sec.pre},
we introduce standard notation and collect preliminary lemmas.
Section~\ref{sec.linear-app} develops a linear approximation argument
and establishes the H\"older continuity of
suitable weak solutions under a smallness condition.
In Section~\ref{sec.theorem-2},
we present the proofs of Theorem \ref{thm.main-U} and Theorem \ref{thm.main-L}.

\section{Preliminaries}\label{sec.pre}

In this preparatory section,
we recall standard notation, definitions, and preliminary results.

\subsection{Notation and definitions}
The following notation and definitions will be used throughout this paper.

For any $(x,t)\in\R\times\R$ and any $r>0$,
we define the biparabolic cylinder
\begin{equation}\label{def.Q-r}
Q_r(x,t):=B_r(x)\times(t-r^4,t+r^4)=(x-r,x+r)\times(t-r^4,t+r^4).
\end{equation}
For simplicity, we write $Q_r:=Q_r(0,0)$ and $B_r:=B_r(0)$.

We write $\partial_t$ for the time derivative and $\partial_x$ to denote the spatial derivative.

We denote by $\Z_{\geq k}:=\{k,k+1,k+2,\dots\}$
the set of integers greater than or equal to $k$.

The symbol $C$ denotes a positive absolute constant,
whose value may vary from line to line.
We write $a\lesssim b$ to indicate that $a\leq Cb$
for some positive absolute constant $C>0$,
and $a\lesssim_m b$ to mean $a\leq C(m)b$
for some positive constant $C(m)>0$
depending only on the parameter $m$.

We use standard notation for Sobolev spaces.
In particular, $H^1(B_1)=W^{1,2}(B_1)$ and $H^2(B_1)=W^{2,2}(B_1)$,
where $W^{k,p}(B_1)$ denotes the Sobolev space consisting of functions on $B_1$
with weak derivatives up to order $k$ in $L^p(B_1)$.

Given a locally integrable function $g$,
we denote its average over $B_r$ and $Q_r$ by
\[
g_{B_r}(t):=\fint_{B_r}g=\frac{1}{|B_r|}\int_{B_r}g(x,t)dx,
\quad
g_{Q_r}:=\fint_{Q_r}g=\frac{1}{|Q_r|}\int_{Q_r}g(x,t)dxdt.
\]
When no confusion arises, we write
\[
\int_{Q_r}|u|^2=\int_{Q_r}|u(x,t)|^2dxdt,
\quad
\sup_{(-r^4,r^4)}\int_{B_r}|u|^2=\sup_{t\in(-r^4,r^4)}\int_{B_r}|u(x,t)|^2dx.
\]

We now recall the notion of suitable weak solutions to \eqref{eqn.SGM-f}.

\begin{definition}[Suitable weak solution]\label{def.sws}
A function $u$ is called a suitable weak solution of \eqref{eqn.SGM-f},
if the following conditions hold:
\begin{enumerate}[\quad(1)]
\item
$u\in L^\infty(-1,1;L^2(B_1))\cap L^2(-1,1;H^2(B_1))$;

\item
$u$ satisfies \eqref{eqn.SGM-f} in the sense of distributions, i.e.,
\begin{equation}\label{eqn.weak-form}
\int_{Q_1}(u\vp_t-u_{xx}\vp_{xx}-u_x^2\vp_{xx}+f\vp)=0
\end{equation}
for any $\vp\in C_c^\infty(Q_1)$, where $f\in L^p(Q_1)$ with $p>3/2$;

\item
the local energy inequality holds:
\begin{align}
\frac{1}{2}\int_{B_1}u^2\vp+\int_{-1}^t\int_{B_1}u_{xx}^2\vp
\leq\int_{-1}^t\int_{B_1}&\bigg(\frac{1}{2}(\vp_t-\vp_{xxxx})u^2+2u_x^2\vp_{xx}
\nonumber
\\
&-\frac{5}{3}u_x^3\vp_x-u_x^2u\vp_{xx}+fu\vp\bigg)
\label{eqn.local-u}
\end{align}
for any nonnegative $\vp\in C_c^\infty(Q_1)$ and almost every $t\in(-1,1)$.
\end{enumerate}
\end{definition}

By choosing appropriate test functions in \eqref{eqn.weak-form},
we obtain
\begin{equation}\label{eqn.weak}
\int_{B_1}(u(s)\vp(s)-u(s')\vp(s'))
=\int_{s'}^{s}\int_{B_1}(u\vp_t-u_{xx}\vp_{xx}-u_x^2\vp_{xx}+f\vp)
\end{equation}
for any $\vp\in C_c^\infty(Q_1)$ and almost every $-1\leq s'<s\leq1$.
This identity, which is an equivalent form of \eqref{eqn.weak-form},
satisfies the condition for applying the parabolic Poincar\'e inequality (Lemma \ref{lem.ppi}).

Moreover,
given a constant $a\in\R$,
one can check that if $u$ is a suitable weak solution of \eqref{eqn.SGM-f},
then $u-a$ also satisfies \eqref{eqn.weak-form} and the same local energy inequality
\begin{align}
\frac{1}{2}\int_{B_1}(u-a)^2\vp+\int_{-1}^t\int_{B_1}u_{xx}^2\vp
\leq\int_{-1}^t\int_{B_1}&\bigg(\frac{1}{2}(\vp_t-\vp_{xxxx})(u-a)^2
+2u_x^2\vp_{xx}
\nonumber
\\
&-\frac{5}{3}u_x^3\vp_x-u_x^2(u-a)\vp_{xx}+f(u-a)\vp\bigg).
\label{eqn.local-u-a}
\end{align}
This follows directly from \eqref{eqn.local-u} and \eqref{eqn.weak}.
Specifically, multiplying \eqref{eqn.weak} by $-a$ gives
\begin{equation}\label{eqn.weak-local}
-a\int_{B_1}u\vp
=-a\int_{-1}^t\int_{B_1}(\vp_t-\vp_{xxxx})u
 +a\int_{-1}^t\int_{B_1}u_x^2\vp_{xx}-a\int_{-1}^t\int_{B_1}f\vp.
\end{equation}
Here, we take $s'=-1$ and $s=t$ in \eqref{eqn.weak}, with $\vp(x,-1)=0$.
Adding \eqref{eqn.local-u} and \eqref{eqn.weak-local} then yields \eqref{eqn.local-u-a}.

\begin{definition}[Biparabolic Hausdorff measure]\label{def.phm}
Given a set $E\subset\R\times\R$,
for any $k\geq0$ we define the $k$-dimensional biparabolic Hausdorff measure by
\[
\mathcal P^k(E):=\lim_{\d\to0^+}\mathcal P^k_\d(E),
\]
where
\[
\mathcal P^k_\d(E):=\inf\set{\sum_{i=1}^\infty r_i^k:E\subset\bigcup_{i=1}^\infty Q_{r_i},\ r_i<\d}.
\]
\end{definition}

We shall use the following lemma to estimate the quantity $O_r$ in the subsequent analysis.

\begin{lemma}\label{lem.mean-int}
Suppose $\Om\subset\R^n$ is a bounded domain and $1\leq p\leq+\infty$.
Then
\[
\norm{g-g_\Om}_{L^p(\Om)}\leq2\norm{g-c}_{L^p(\Om)}
\]
for any $g\in L^p(\Om)$ and any $c\in\R$, where $g_\Om=\fint_\Om g$.
In particular, when $p=2$, we have
\[
\norm{g-g_\Om}_{L^2(\Om)}\leq\norm{g-c}_{L^2(\Om)}
\]
for any $g\in L^2(\Om)$ and any $c\in\R$.
\end{lemma}
The proof of Lemma \ref{lem.mean-int} is elementary and is therefore omitted.

\subsection{Local energy estimate}
We derive a local energy estimate from
the local energy inequality by choosing an appropriate test function
and applying Young's and H\"older's inequalities.

\begin{lemma}[Local energy estimate]
Let $u$ be a suitable weak solution to \eqref{eqn.SGM-f}
with $f\in L^p(Q_1)$.
Then
\[
\sup_{(-1/2^4,1/2^4)}\int_{B_{1/2}}u^2+\int_{Q_{1/2}}u_{xx}^2
\ls_p\int_{Q_1}(u^2+|u|^3+u_x^2+|u_x|^3)+\frb{\int_{Q_1}|f|^p}^{3/(2p)}.
\]
\end{lemma}

\begin{proof}
Let $\vp$ be a smooth cutoff function satisfying $\vp=1$ on $Q_{1/2}$,
$0\leq\vp\leq1$ in $Q_{3/4}$, and $\vp=0$ outside $Q_{3/4}$.
Using $\vp$ as a test function in \eqref{eqn.local-u}
and applying Young's and H\"older's inequalities,
we directly obtain the desired estimate.
\end{proof}

The above estimate can be further simplified by observing that
the $L^2$-norm can be controlled by $L^3$-norm.
Specifically, we often use the following simplified form
\begin{equation}\label{eqn.local-energy}
\sup_{(-1/2^4,1/2^4)}\int_{B_{1/2}}u^2+\int_{Q_{1/2}}u_{xx}^2
\ls_p\int_{Q_1}(|u|^3+|u_x|^3)+\frb{\int_{Q_1}|f|^p}^{3/(2p)}.
\end{equation}
This estimate will be particularly useful in our later analysis.

\begin{lemma}\label{lem.suit-es}
Assume that
\[
\int_{Q_1}(|u|^3+|u_x|^3)+\frb{\int_{Q_1}|f|^p}^{3/(2p)}\ls1.
\]
Then
\begin{align*}
\int_{Q_{1/2}}(|u|^{10/3}+|u_x|^{10/3})&\ls_p 1.
\end{align*}
\end{lemma}

\begin{proof}
By the local energy estimate \eqref{eqn.local-energy}, we have
\[
\sup_{(-1/2^4,1/2^4)}\int_{B_{1/2}}u^2
\ls_p\int_{Q_1}(|u|^3+|u_x|^3)+\frb{\int_{Q_1}|f|^p}^{3/(2p)}\ls_p 1,
\]
and
\[
\int_{-1/2^4}^{1/2^4}\norm{u}_{H^2(B_{1/2})}^2\ls_p 1
+\frb{\int_{Q_1}|u_x|^3}^{2/3}
\ls_p 1.
\]
Hence
\begin{align*}
\int_{Q_{1/2}}|u|^{10/3}
&\leq\int_{-1/2^4}^{1/2^4}\frb{\int_{B_{1/2}}|u|^2}^{2/3}\frb{\int_{B_{1/2}}|u|^6}^{1/3}
\\
&\leq\frb{\sup_{(-1/2^4,1/2^4)}\int_{B_{1/2}}|u|^2}^{2/3}
\int_{-1/2^4}^{1/2^4}\norm{u}_{L^6(B_{1/2})}^2
\\
&\ls_p \frb{\sup_{(-1/2^4,1/2^4)}\int_{B_{1/2}}|u|^2}^{2/3}
\int_{-1/2^4}^{1/2^4}\norm{u}_{H^2(B_{1/2})}^2\ls_p 1,
\end{align*}
where we used the Sobolev embedding $H^2(B_{1/2})\subset L^6(B_{1/2})$ in the third inequality.

Similarly,
by the one-dimensional Gagliardo--Nirenberg inequality,
for all $t\in(-1/2^4,1/2^4)$,
\[
\norm{\p_xu(\cdot,t)}_{L^{10/3}(B_{1/2})}
\ls\norm{u(\cdot,t)}_{H^2(B_{1/2})}^{3/5}\norm{u(\cdot,t)}_{L^2(B_{1/2})}^{2/5}.
\]
Thus
\begin{align*}
\int_{Q_{1/2}}|u_x|^{10/3}
&\ls\int_{-1/2^4}^{1/2^4}\norm{u}^2_{H^2(B_{1/2})}\norm{u}_{L^2(B_{1/2})}^{4/3}
\\
&\ls\frb{\sup_{(-1/2^4,1/2^4)}\int_{B_{1/2}}|u|^2}^{2/3}
\int_{-1/2^4}^{1/2^4}\norm{u}^2_{H^2(B_{1/2})}\ls_p1.
\end{align*}
This completes the proof.
\end{proof}

\subsection{Parabolic Poincar\'e inequality}
In this subsection,
we present a parabolic Poincar\'e inequality.
Compared with the classical Poincar\'e inequality,
it has the notable advantage that no bounds on time derivatives are required.

\begin{lemma}[Parabolic Poincar\'e inequality]\label{lem.ppi}
Let $r\in(0,1]$, $\vartheta\in[0,1]$, and $p>3/2$.
Let $u\in L^3(Q_r(x_0,t_0))$ with $u_x\in L^3(Q_r(x_0,t_0))$
and $f\in L^p(Q_r(x_0,t_0))$. If
\begin{equation}\label{eqn.ppi}
\int_{B_r(x_0)}(u(t)-u(s))\phi
=\int_s^t\int_{B_r(x_0)}u_x\phi_{xxx}-\vartheta\int_s^t\int_{B_r(x_0)}u_x^2\phi_{xx}
+\int_s^t\int_{B_r(x_0)}f\phi
\end{equation}
for all $\phi\in C_c^\infty(B_r(x_0))$
and almost every $s,t\in(-r^4,r^4)$ with $s<t$, then
\begin{align*}
&\quad\,\fint_{Q_{r/2}(x_0,t_0)}|u-u_{Q_{r/2}(x_0,t_0)}|^3
\\
&\ls r^3\fint_{Q_r(x_0,t_0)}|u_x|^3+\vartheta\frb{r^3\fint_{Q_r(x_0,t_0)}|u_x|^3}^2
+\frb{r^4\frb{\fint_{Q_r(x_0,t_0)}|f|^p}^{1/p}}^3.
\end{align*}
\end{lemma}

\begin{proof}
The proof follows the same line as in \cite[Theorem 3.1]{OR19},
except for the presence of the forcing term,
which can be handled analogously.
For completeness, we provide a brief sketch.
Without loss of generality, we may assume that $(x_0,t_0)=(0,0)$.
Let $\s:\R\to[0,1]$ be a smooth cut-off function
such that $\s=1$ in $B_{r/2}$ and $\s=0$ in $B_r^c$,
with $|\p_x^k\s|\ls_k r^{-k}$ for all $k\in\Z_+$.
Set
\[
u_r^\s(t):=\frac{\int_{B_r}u(t)\s}{\int_{B_r}\s},
\]
and consider the test function
\[
\phi(x):=(u_r^\s(t)-u_r^\s(s))|u_r^\s(t)-u_r^\s(s)|\s(x)\in C_c^\infty(B_r).
\]
Thus
\begin{align*}
\int_s^t\int_{B_r}f\phi
&\leq|u_r^\s(t)-u_r^\s(s)|^2\int_s^t\int_{B_r}|f||\s|
\\
&\leq|u_r^\s(t)-u_r^\s(s)|^2\frb{\int_{Q_r}|f|^p}^{1/p}\frb{\int_{Q_r}1}^{1-1/p}
\\
&=C|u_r^\s(t)-u_r^\s(s)|^2r^5\frb{\fint_{Q_r}|f|^p}^{1/p}
\\
&=C|u_r^\s(t)-u_r^\s(s)|^2r^{2/3} r^{13/3}\frb{\fint_{Q_r}|f|^p}^{1/p}
\\
&\leq C\ve|u_r^\s(t)-u_r^\s(s)|^3r+C_\ve r\frb{r^4\frb{\fint_{Q_r}|f|^p}^{1/p}}^3,
\end{align*}
where we used Young's inequality with $\ve>0$ in the last inequality.
All remaining steps proceed exactly as in \cite[Theorem 3.1]{OR19}
and are therefore omitted.
\end{proof}

\subsection{Compactness of suitable weak solutions}

This subsection is devoted to establishing the strong convergence
of a sequence of suitable weak solutions to \eqref{eqn.SGM-f}
by applying the Aubin--Lions lemma,
which will be used in the proof of the approximation lemma (Lemma \ref{lem.appro}).

\begin{lemma}
[Aubin--Lions lemma,
\texorpdfstring{\cite[Theorem 4.12]{RRS16}}{}]
Suppose that $X_0\subset X\subset X_1$ where $X_0$, $X$,
and $X_1$ are reflexive Banach spaces
and the embedding $X_0\subset X$ is compact.
If the sequence $u_k$ is bounded in $L^q(0,T;X_0)$, $q\geq1$,
and $\p_tu_k$ is bounded in $L^p(0,T;X_1)$, $p\geq1$,
then there exists a subsequence of $u_k$ that is strongly convergent in $L^q(0,T;X)$.
\end{lemma}

\begin{lemma}\label{lem.appro}
Assume that $\set{u_k}_{k=1}^\infty$ is a sequence of
suitable weak solutions to \eqref{eqn.SGM-f}
with $f_k\in L^p(Q_1)$ satisfying
\[
\int_{Q_1}|u_k|^3+\int_{Q_1}|\p_xu_k|^3+\frb{\int_{Q_1}|f_k|^p}^{1/p}\ls1
\quad  \text{for all}\ k\in\Z_+.
\]
Then the sequence $\set{\p_xu_k}_{k=1}^\infty$ converges strongly in $L^3(Q_{1/2})$.
\end{lemma}

\begin{proof}
By the local energy estimate, it holds that
\[
\sup_{(-1/2^4,1/2^4)}\int_{B_{1/2}}u_k^2+\int_{Q_{1/2}}|\p_{xx}u_k|^2\ls_p1.
\]
By lemma \ref{lem.suit-es}, we have
\begin{equation}\label{eqn.10-3-u_k}
\int_{Q_{1/2}}|u_k|^{10/3}+\int_{Q_{1/2}}|\p_xu_k|^{10/3}\ls_p1.
\end{equation}

Next, for any $\vp\in C_c^\infty(Q_{1/2})$, we conclude that
\begin{align*}
\int_{Q_{1/2}}u_k\vp_t
&\leq\int_{Q_{1/2}}|\p_{xx}u_k||\vp_{xx}|
+\int_{Q_{1/2}}|\p_xu_k|^2|\vp_{xx}|+\int_{Q_{1/2}}|f_k||\vp|
\\
&\ls_p(\norm{\p_{xx}u_k}_{L^{3/2}(Q_{1/2})}
+\norm{\p_xu_k}^2_{L^3(Q_{1/2})}
+\norm{f_k}_{L^p(Q_{1/2})})\norm{\vp}_{L^3(-1/2^4,1/2^4;W_0^{2,3}(B_{1/2}))}
\\
&\ls_p\norm{\vp}_{L^3(-1/2^4,1/2^4;W_0^{2,3}(B_{1/2}))}.
\end{align*}
This implies that
$\set{\p_tu_k}_{k=1}^\infty\subset L^{3/2}(-1/2^4,1/2^4;W^{-2,3/2}(B_{1/2}))$.

Set
\[
X_0:=H^2(B_{1/2})\subset\subset X:=W^{1,3/2}(B_{1/2})\subset X_1:=W^{-2,3/2}(B_{1/2}).
\]
By the Aubin--Lions lemma,
it follows that $\set{u_k}_{k=1}^\infty$ is strongly convergent in
\[
L^{3/2}(-1/2^4,1/2^4;W^{1,3/2}(B_{1/2})).
\]
Combining this with \eqref{eqn.10-3-u_k},
we deduce that $\set{\p_xu_k}_{k=1}^\infty$ converges strongly in $L^3(Q_{1/2})$.
\end{proof}

\section{H\"older continuity via the compactness method}\label{sec.linear-app}

In this section,
we aim to establish the H\"older continuity of suitable weak solutions to \eqref{eqn.SGM-f}
under a smallness condition on the quantities $G_r$ and $F_r$
by employing the compactness method.
This method relies on a linear approximation argument,
in which suitable weak solutions of \eqref{eqn.SGM-f}
are approximated by weak solutions of the linear biharmonic heat equation
\begin{equation}\label{eqn.linear}
v_t+v_{xxxx}=0\quad \text{in}\ Q_{1/4}.
\end{equation}
The following approximation lemma provides a detailed account of this linear approximation.
Before proceeding, we recall the precise definition of
weak solutions to \eqref{eqn.linear}.
Specifically, a function $v\in L^2(Q_{1/4})$ with $v_x\in L^2(Q_{1/4})$
is called a weak solution of \eqref{eqn.linear}
if
\[
\int_{Q_{1/4}}v\vp_t=\int_{Q_{1/4}}v\vp_{xxxx}
\]
for every test function $\vp\in C_c^\infty(Q_{1/4})$.

\subsection{Key approximation lemma}

We now present the key approximation lemma,
which is central to the compactness method.

\begin{lemma}[Approximation lemma]\label{lem.appro}
For any $\ve>0$,
there is a $\d=\d(\ve,p)>0$ such that
for any suitable weak solution $u$ of \eqref{eqn.SGM-f} with $f \in L^p(Q_1)$
satisfying $G_1[u]+F_1[f]^{1/2}\leq\d$,
there exists a weak solution $v$ of \eqref{eqn.linear}
with $G_\frac{1}{4}[v]\leq3(G_1[u]+F_1[f]^{1/2})$, such that
\[
G_\frac{1}{4}[u-v]\leq\ve(G_1[u]+F_1[f]^{1/2}).
\]
\end{lemma}

\begin{proof}
We argue by contradiction.
Assume that the lemma is false.
Then there exists $\ve_0>0$ such that for any $\d=1/k\ (k\in\Z_+)$,
there exists a sequence of suitable weak solutions $\set{u_k}_{k=1}^\infty$
to \eqref{eqn.SGM-f} with $f_k\in L^p(Q_1)$
satisfying $G_1[u_k]+F_1[f_k]^{1/2}\leq1/k$,
such that for any weak solution $v$ of \eqref{eqn.linear}
with $G_\frac{1}{4}[v]\leq3(G_1[u_k]+F_1[f_k]^{1/2})$,
the following inequality holds
\[
G_\frac{1}{4}[u_k-v]>\ve_0(G_1[u_k]+F_1[f_k]^{1/2}).
\]

Set
\[
v_k:=\frac{u_k-(u_k)_{Q_{1/2}}}{G_1[u_k]+F_1[f_k]^{1/2}}
\quad \text{and} \quad
g_k:=\frac{f_k}{G_1[u_k]+F_1[f_k]^{1/2}}.
\]
Then $v_k$ satisfies
\begin{equation}\label{eqn.app}
\p_tv_k+\p_{xxxx}v_k+(G_1[u_k]+F_1[f_k]^{1/2})\p_{xx}(\p_xv_k)^2=g_k \quad \text{in}\ Q_1,
\end{equation}
such that for any weak solution $v$ of \eqref{eqn.linear} with $G_\frac{1}{4}[v]\leq3$,
it holds that
\[
G_\frac{1}{4}[v_k-v]>\ve_0.
\]

Since
\begin{align*}
F_1[g_k]=\frac{F_1[f_k]}{G_1[u_k]+F_1[f_k]^{1/2}}\leq F_1[f_k]^{1/2}\leq1/k,
\end{align*}
we have $g_k\to0$ strongly in $L^p(Q_1)$.
Observe that $\fint_{Q_{1/2}}v_k=0$ and
\begin{equation}\label{eqn.G-v-k}
G_1[v_k]=\frb{\fint_{Q_1}|\p_xv_k|^3}^{1/3}\leq1.
\end{equation}
By the parabolic Poincar\'e inequality (Lemma~\ref{lem.ppi}), it follows that
\begin{equation}\label{eqn.v_k-G-F}
\int_{Q_{1/2}}|v_k|^3\ls G_1[v_k]^3+(G_1[u_k]+F_1[f_k]^{1/2})G_1[v_k]^6
+F_1[g_k]^3\ls1.
\end{equation}
Thus, both sequences $\set{v_k}_{k=1}^\infty$
and $\set{\p_xv_k}_{k=1}^\infty$ are bounded in $L^3(Q_{1/2})$.
Consequently, there exists a function $v_\infty\in L^3(Q_{1/2})$
with $\p_xv_\infty\in L^3(Q_{1/2})$
such that, up to a subsequence,
\[
v_k\to v_\infty,\ \ \p_xv_k\to \p_xv_\infty \quad
\text{weakly in}\ L^3(Q_{1/2})\ \ \text{as}\ \, k\to\infty.
\]
Moreover, due to \eqref{eqn.G-v-k}, \eqref{eqn.v_k-G-F}, and $F_1[g_k]\leq1/k$,
we may apply Lemma~\ref{lem.appro} to obtain
\[
\p_xv_k\to\p_xv_\infty\quad \text{strongly in}\ L^3(Q_{1/4})\ \ \text{as}\ \, k\to\infty.
\]
Combining this strong convergence with $G_1[v_k]\leq1$ yields
\[
G_\frac{1}{4}[v_\infty]=\lim_{k\to\infty}G_\frac{1}{4}[v_k]
\leq\frac{1}{4}4^{5/3}\lim_{k\to\infty}G_1[v_k]\leq3.
\]
Taking the limit as $k\to\infty$ in \eqref{eqn.app},
we find that $v_\infty$ is a weak solution of \eqref{eqn.linear}.
However,
\[
0<\ve_0<G_\frac{1}{4}[v_k-v_\infty]\to0\ \, (k\to\infty),
\]
which leads to a contradiction and thus completes the proof.
\end{proof}

\begin{remark}
Due to the invariance property \eqref{eqn.invarp} of $G_r$
and the parabolic Poincar\'e inequality,
we can avoid using the following control functional
\[
\frb{\fint_{Q_1}|u|^3}^{1/3}+G_1[u]+F_1[f]^{1/2}.
\]
\end{remark}

\subsection{Contraction lemma and decay lemma}

Based on the approximation lemma,
we derive a contraction lemma
that provides control over the quantity $G_r$ at smaller scales.
This lemma also relies on a local $L^\infty$ estimate for weak solutions of \eqref{eqn.linear}.
By iterating the contraction lemma, we further obtain a decay lemma,
which implies the H\"older continuity of suitable weak solutions via a Campanato-type argument.
In addition,
the decay of $F_r$ can be effectively controlled by the following lemma.

\begin{lemma}\label{lem.Frk}
Let $f\in L^p(Q_1)$ with $p>3/2$.
Then there exists a universal constant $0<\eta_p<1$
such that, for any $r\in(0,\eta_p]$ and any $k\in\Z_{\geq0}$,
\begin{equation}\label{eqn.Frk}
F_{r^k}[f]\leq\frac{1}{4^{2k}}F_1[f].
\end{equation}
\end{lemma}

\begin{proof}
Set $\eta_p:=4^{-2p/(2p-3)}\in(0,1)$.
Then for any $r\in(0,\eta_p]$ and any $k\in\Z_+$,
\[
F_{r^k}[f]
\leq r^{4k-5k/p}F_1[f]\leq\eta_p^{4k-5k/p}F_1[f]
\leq\eta_p^{2k-3k/p}F_1[f]=\frac{1}{4^{2k}}F_1[f].
\]
\end{proof}

We remark that Lemma \ref{lem.Frk} is inspired by \cite[Lemma 3.1]{DHZ25}.

\begin{lemma}
[$L^\infty$ estimate,
\texorpdfstring{\cite[Proposition 4.1]{OR19}}{}]
\label{lem.L-infty}
Let $v\in L^2(Q_r)$ with $v_x\in L^2(Q_r)$ be
a weak solution of \eqref{eqn.linear} in $Q_r$, where $r>0$. Then
\[
\norm{v_x}_{L^\infty(Q_{r/2})}\ls_r\norm{v_x}_{L^2(Q_r)}+\norm{v}_{L^2(Q_r)}.
\]
\end{lemma}

\begin{lemma}[Contraction lemma]\label{lem.Contra}
Let $\eta_p$ be the constant from Lemma~\ref{lem.Frk}.
There exist universal constants $\d_0=\d_0(p)>0$ and $0<\l=\l(p)<\eta_p<1$ such that
if $u$ is a suitable weak solution of \eqref{eqn.SGM-f}
with $f \in L^p(Q_1)$ satisfying
$G_1[u]+F_1[f]^{1/2}\leq\d_0$, then
\[
G_\l[u]\leq\frac{1}{4}G_1[u]+\frac{1}{4}F_1[f]^{1/2}.
\]
\end{lemma}

\begin{proof}
Let $\ve>0$ in the approximation lemma be fixed later.
Then there exists $\d=\d(\ve,p)>0$ such that if $u$ is a suitable weak solution of
\eqref{eqn.SGM-f} with $f\in L^p(Q_1)$ satisfying $G_1[u]+F_1[f]^{1/2}\leq\d$,
then there exists a weak solution $v$ of \eqref{eqn.linear}
with $G_\frac{1}{4}[v]\ls G_1[u]+F_1[f]^{1/2}$,
such that
\[
G_\frac{1}{4}[u-v]\leq\ve(G_1[u]+F_1[f]^{1/2}).
\]

Note that $v-v_{Q_{1/8}}$ remains a weak solution of \eqref{eqn.linear}.
Using the $L^\infty$ estimate (Lemma \ref{lem.L-infty}),
\begin{align*}
\norm{v_x}_{L^\infty(Q_{1/16})}
&\ls\norm{v_x}_{L^2(Q_{1/8})}+\norm{v-v_{Q_{1/8}}}_{L^2(Q_{1/8})}
\\
&\ls\norm{v_x}_{L^3(Q_{1/4})}+\norm{v-v_{Q_{1/8}}}_{L^3(Q_{1/8})}
\\
&\ls\norm{v_x}_{L^3(Q_{1/4})}\ls G_\frac{1}{4}[v]\ls G_1[u]+F_1[f]^{1/2},
\end{align*}
where the third inequality follows from
the parabolic Poincar\'e inequality with $\vartheta=0$ and no forcing term.

Consequently, for any $0<r\leq1/16$,
\begin{align*}
G_r[u]&=r\frb{\fint_{Q_r}|u_x|^3}^{1/3}
\leq r\frb{\fint_{Q_r}|u_x-v_x|^3}^{1/3}+r\frb{\fint_{Q_r}|v_x|^3}^{1/3}
\\
&\leq r(4r)^{-5/3}\frb{\fint_{Q_\frac{1}{4}}|u_x-v_x|^3}^{1/3}+r\norm{v_x}_{L^\infty(Q_{1/16})}
\\
&\leq (4r)^{-2/3}\ve(G_1[u]+F_1[f]^{1/2})+Cr(G_1[u]+F_1[f]^{1/2}).
\end{align*}
Take $0<r=\l<\min\set{1/16,\eta_p}$ sufficiently small such that $C\l<1/8$,
and then choose $\ve>0$ small enough to ensure $(4\l)^{-2/3}\ve <1/8$.
Hence, $\d_0:=\d(p)$ is fixed, which completes the proof.
\end{proof}

The following lemma is derived by iterating the contraction lemma.

\begin{lemma}[Decay lemma]\label{lem.Iterative}
Let $\d_0$ and $\l$ be the constants from the contraction lemma.
If $u$ is a suitable weak solution of \eqref{eqn.SGM-f}
with $f \in L^p(Q_1)$ satisfying $G_1[u]+F_1[f]^{1/2}\leq\d_0$,
then for any $k\in\Z_{\geq0}$,
\begin{equation}\label{eqn.G-l-k}
G_{\l^k}[u]\leq\frac{1}{4^k}G_1[u]+\frac{k}{4^k}F_1[f]^{1/2}.
\end{equation}
\end{lemma}

\begin{proof}
We prove this lemma by induction.
The case $k=0$ is trivial,
and the case $k=1$ is exactly the contraction lemma.
Assume that \eqref{eqn.G-l-k} holds for some $k\in\Z_{\geq2}$.
We now show it also holds for $k+1$.
To this end,
consider the rescaled functions $u_{k+1}(x,t):=u(\l^kx,\l^{4k}t)$
and $f_{k+1}(x,t):=\l^{4k}f(\l^kx,\l^{4k}t)$, which satisfy
\begin{align*}
G_1[u_{k+1}]+F_1[f_{k+1}]^{1/2}
&=G_{\l^k}[u]+F_{\l^k}[f]^{1/2}
\leq\frac{1}{4^k}G_1[u]+\frac{k}{4^k}F_1[f]^{1/2}+\frac{1}{4^k}F_1[f]^{1/2}
\\
&=\frac{1}{4^k}G_1[u]+\frac{k+1}{4^k}F_1[f]^{1/2}
\leq\frac{k+1}{4^k}\d_0\leq\d_0,
\end{align*}
where we used Lemma \ref{lem.Frk} in the first inequality.

Applying the contraction lemma to $u_{k+1}$ gives
\[
G_\l[u_{k+1}]\leq\frac{1}{4}G_1[u_{k+1}]+\frac{1}{4}F_1[f_{k+1}]^{1/2}.
\]
Rescaling back yields
\[
G_{\l^{k+1}}[u]\leq\frac{1}{4^{k+1}}G_1[u]+\frac{k+1}{4^{k+1}}F_1[f]^{1/2}.
\]
This completes the induction argument and the proof.
\end{proof}

\subsection{Proof of Theorem \ref{thm.main-G+F}}

We now present the proof of Theorem \ref{thm.main-G+F}.

\begin{lemma}[Campanato lemma,
\texorpdfstring{\cite[Lemma A.2]{OR19}}{}]
\label{lem.Campanato}
Let $g\in L^1(Q_1)$ and
suppose that there exist positive constants $\a\in(0,1]$ and $C>0$,
such that
\[
\frb{\fint_{Q_r(x,t)}|g-g_{Q_r(x,t)}|^3}^{1/3}\leq Cr^\a
\]
for any $(x,t)\in Q_{1/2}$ and any $r\in(0,1/2)$.
Then $g$ is H\"older continuous in $Q_{1/2}$ with
\[
|g(y,s)-g(y',s')|\leq C(|y-y'|+|s-s'|^{1/4})^\a
\]
for all $(y,s),(y',s')\in Q_{1/2}$.
\end{lemma}

\begin{proposition}\label{pro.u-C^a}
There exist universal constants $\d_0=\d_0(p)>0$ and $0<\a=\a(p)<1$ such that
if $u$ is a suitable weak solution of
\eqref{eqn.SGM-f} with $f \in L^p(Q_1)$ satisfying
\[
G_1[u]+F_1[f]^{1/2}\leq\frac{\d_0}{2},
\]
then $u\in C^{\a,\a/4}(Q_{1/4})$.
\end{proposition}

\begin{proof}
Given $(x_0,t_0)\in Q_{1/4}$, we have $Q_{1/2}(x_0,t_0)\subset Q_1$.
Consider the rescaled functions
\[
\wt u(x,t):=u(2^{-1}x+x_0,2^{-4}t+t_0),\quad
\wt f(x,t):=2^{-4}f(2^{-1}x+x_0,2^{-4}t+t_0),
\]
which satisfy
\[
G_1[\wt u]+F_1[\wt f]^{1/2}\leq2^{2/3}G_1[u]+2^{2/3}F_1[f]^{1/2}\leq\d_0,
\]
where $\d_0$ is the constant from the decay lemma.

By the decay lemma (Lemma \ref{lem.Iterative}),
for any $k\in\Z_{\geq0}$,
\[
G_{\l^k}[\wt u]\leq\frac{1}{4^k}G_1[\wt u]+\frac{k}{4^k}F_1[\wt f]^{1/2}
\leq\frac{1+k}{4^k}\d_0\ls_p\frac{1}{2^k}.
\]
Rescaling back yields
\[
\int_{Q_{\l^k/2}(x_0,t_0)}|u_x|^3\ls_p\l^{2k}2^{-3k}.
\]

For any $0<r<1/2$,
there exists $k\in\Z_{\geq0}$ such that $\l^{k+1}/2\leq r<\l^k/2$,
and thus
\[
2^{-(k+1)}<2^{-\log_{\l}2r}=(2r)^{-\frac{1}{\log_2\l}}\ls_p r^\a,
\quad 0<\a:=-\frac{1}{\log_2\l}<1.
\]
It then follows that
\[
r^3\fint_{Q_r{(x_0,t_0)}}|u_x|^3\ls_p \l^{-2k}
\int_{Q_{\l^k/2}{(x_0,t_0)}}|u_x|^3\ls_p2^{-3k}\ls_p r^{3\a}.
\]
Moreover, since $\l<\eta_p$, we have
\[
r^4\frb{\fint_{Q_r(x_0,t_0)}|f|^p}^{1/p}
\leq r^{4-5/p}F_1[f]
\ls_p \l^{k(4-5/p)}\ls_p 2^{-k}\ls_p r^{\a}.
\]
Therefore, the parabolic Poincar\'e inequality (Lemma~\ref{lem.ppi}) implies
\begin{align*}
\fint_{Q_{r/2}(x_0,t_0)}|u-u_{Q_{r/2}(x_0,t_0)}|^3
\ls_p r^{3\a}+(r^{3\a})^2+r^{3\a}\ls_p r^{3\a}.
\end{align*}
By the Campanato lemma,
we conclude that $u\in C^{\a,\a/4}(Q_{1/4})$.
\end{proof}

\begin{proof}[Proof of Theorem \ref{thm.main-G+F}]
Let $\d_0$ and $\a$ be the constants from Proposition \ref{pro.u-C^a}.
Consider the rescaled functions $u^r(x,t)=u(rx,r^4t)$ and $f^r(x,t)=r^4f(rx,r^4t)$,
which satisfy
\[
G_1[u^r]+F_1[f^r]^{1/2}=G_r[u]+F_r[f]^{1/2}\leq\frac{\d_0}{2}.
\]
By Proposition \ref{pro.u-C^a},
we have $u^r\in C^{\a,\a/4}(Q_{1/4})$,
and hence $u\in C^{\a,\a/4}(Q_{r/4})$.
\end{proof}

\section{Proofs of Theorem \ref{thm.main-U} and Theorem \ref{thm.main-L}}\label{sec.theorem-2}

\subsection{Interpolation inequalities}

This subsection presents two interpolation inequalities
that play a key role in the subsequent analysis of the regularity criteria.

\begin{lemma}[Interpolation inequalities]
For any $0<r\leq1$, it holds that
\begin{align}
\fint_{Q_r}|u|^3&\ls U_r^{9/7}G_r^{3/7}+U_r^{3/2}, \label{eqn.U-G-U}
\\
G_r&\ls U_r^{5/24}L_r^{7/24}+U_r^{1/2}.  \label{eqn.U-L-U}
\end{align}
\end{lemma}

\begin{proof}
Due to the scale-invariance, we may assume that $r=1$.

(1)
By the Gagliardo--Nirenberg inequality (\cite[Theorem 7.41]{L17}),
we get
\[
\norm{u}_{L^3(B_1)}^3\ls
\norm{u_{x}}_{L^3(B_1)}^{3/7}\norm{u}_{L^2(B_1)}^{18/7}+\norm{u}^3_{L^2(B_1)}
\]
Integrating this inequality in time over $(-1,1)$ gives
\begin{align*}
\fint_{Q_1}|u|^3&
\ls\int_{-1}^{1}\norm{u_{x}}_{L^3(B_1)}^{3/7}\norm{u}_{L^2(B_1)}^{18/7}
+\int_{-1}^{1}\norm{u}^3_{L^2(B_1)}
\\
&\ls U_1^{18/14}\frb{\int_{-1}^{1}\norm{u_{x}}_{L^3(B_1)}^3}^{1/7}
+\int_{-1}^{1}\frb{\int_{B_1}u^2}^{3/2}
\\
&\ls U_1^{9/7}G_1^{3/7}+U_1^{3/2}.
\end{align*}
This yields the desired estimate \eqref{eqn.U-G-U}.

(2)
By the Gagliardo--Nirenberg inequality (\cite[Theorem 7.41]{L17}), we have
\[
\norm{u_x}^3_{L^3(B_1)}\ls \norm{u_{xx}}_{L^2(B_1)}^{7/4}\norm{u}_{L^2(B_1)}^{5/4}
+\norm{u}_{L^2(B_1)}^3.
\]
Integrating this inequality in time over $(-1,1)$ gives
\begin{align*}
G_1^3&=\fint_{Q_1}|u_x|^3
\ls\int_{-1}^{1}\norm{u_{xx}}_{L^2(B_1)}^{7/4}\norm{u}_{L^2(B_1)}^{5/4}
+\int_{-1}^{1}\norm{u}^3_{L^2(B_1)}\\
&\ls U_1^{5/8}\frb{\int_{-1}^{1}\norm{u_{xx}}_{L^2(B_1)}^2}^{7/8}+U_1^{3/2}
\ls U_1^{5/8}L_1^{7/8}+U_1^{3/2}.
\end{align*}
Taking the $1/3$-th power of both sides yields \eqref{eqn.U-L-U}.
\end{proof}

\subsection{Regularity criterion in terms of $U_r$}

In this subsection, we aim to establish
the regularity criterion in terms of the quantity $U_r$.
Our proof adapts the approach developed by Wang \cite{W25}
for the Navier--Stokes equations,
which relies on a key nonlinear contraction combining
the local energy estimate with the interpolation inequalities.
This approach allows us to demonstrate
the smallness of the quantity $G_r+F_r^{1/2}$,
thereby ensuring the H\"older continuity of suitable weak solutions to \eqref{eqn.SGM-f},
as stated in Theorem \ref{thm.main-G+F}.

\begin{lemma}[Key nonlinear contraction]\label{lem.key-non-iter}
Let $\eta_p$ be the constant from Lemma~\ref{lem.Frk}.
There exist universal constants $0<\d^*_1=\d_1^*(p)<1$
and $0<\th=\th(p)<\eta_p<1$  such that
if $u$ is a suitable weak solution of \eqref{eqn.SGM-f}
with $f \in L^p(Q_1)$ satisfying $U_1\leq\d^*_1$,
then
\[
G_{\th}\leq\frac{1}{4}G_1+\frac{1}{4}F_1^{1/2}+U_1^{1/4}.
\]
\end{lemma}
\begin{proof}
By the local energy estimate \eqref{eqn.local-energy}
and the interpolation inequality \eqref{eqn.U-G-U}, we obtain
\[
L_{1/2}
\ls_p\int_{Q_1}|u|^3+\int_{Q_1}|u_x|^3+\frb{\int_{Q_1}|f|^p}^{3/2p}
\ls_p U_1^{9/7}G_1^{3/7}+U_1^{3/2}+G_1^3+F_1^{3/2}.
\]
Combining this with \eqref{eqn.U-L-U} and applying Young's inequality yields
\begin{align*}
G_r&\ls U_r^{5/24}L_r^{7/24}+U_r^{1/2}
\ls r^{-1/2} U_1^{5/24}L_{1/2}^{7/24}+r^{-1/2}U_1^{1/2}
\\
&\ls_p r^{-1/2} U_1^{5/24}(U_1^{9/7}G_1^{3/7}+U_1^{3/2}+G_1^3+F_1^{3/2})^{7/24}+r^{-1/2}U_1^{1/2}
\\
&\ls_p r^{-1}U_1^{7/12} r^{1/2}G_1^{1/8}+r^{-1/2}U_1^{31/48}+r^{-1}U_1^{5/24}r^{1/2}G_1^{7/8}
+r^{-1}U_1^{5/24} r^{1/2}F_1^{7/16}+r^{-1/2}U_1^{1/2}
\\
&\ls_p r^{-8/7}U_1^{2/3}+r^4G_1+r^{-1/2}U_1^{31/48}+r^{-8}U_1^{5/3}+r^{4/7}G_1
+r^{-8}U_1^{5/3}+r^{4/7}F_1^{1/2}+r^{-1/2}U_1^{1/2}
\\
&\leq C(r^4+r^{4/7})G_1+Cr^{4/7}F_1^{1/2}+C(r^{-8/7}U_1^{5/12}
+r^{-1/2}U_1^{19/48}+r^{-8}U_1^{17/12}+r^{-1/2}U_1^{1/4})U_1^{1/4}.
\end{align*}
Choose $r=\th<\min\set{\eta_p,1/2}$ sufficiently small such that $C(\th^4+\th^{4/7})<1/4$,
and then take $\d_1^*>0$ small enough to ensure
\begin{align*}
&\quad\,C(\th^{-8/7}U_1^{5/12}+\th^{-1/2}U_1^{19/48}+\th^{-8}U_1^{17/12}+\th^{-1/2}U_1^{1/4})
\\
&\leq C[\th^{-8/7}(\d_1^*)^{5/12}+\th^{-1/2}(\d_1^*)^{19/48}
+\th^{-8}(\d_1^*)^{17/12}+\th^{-1/2}(\d_1^*)^{1/4}]<1.
\end{align*}
This yields the desired result.
\end{proof}

The following results are direct consequences of scaling and
iteration arguments applied to the above lemma.

\begin{lemma}\label{lem.G-th^k}
Let $\d^*_1$ and $\th$ be the constants from Lemma \ref{lem.key-non-iter}.
If $\sup_{0<r\leq1}U_r\leq\d_1^*$, then
\begin{align}
G_{\th r}&\leq\frac{1}{4}G_r+\frac{1}{4}F_r^{1/2}+U_r^{1/4}
\quad \text{for any}\ 0<r\leq1, \label{eqn.G-th-r}
\\
G_{\th^k}&\leq\frac{1}{4^k}G_1+\frac{k}{4^k}F_1^{1/2}
+\sum_{i=0}^{k-1}\frac{1}{4^{k-1-i}}U_{\th^i}^{1/4}
\quad \text{for any}\ k\in\Z_+. \label{eqn.G-th-k}
\end{align}
\end{lemma}

\begin{proof}
Given $0<r\leq1$,
consider the rescaled functions $u^r(x,t)=u(rx,r^4t)$
and $f^r(x,t)=r^4f(rx,r^4t)$, which satisfy
\begin{align*}
U_1[u^r]=U_r[u]\leq\sup_{0<r\leq1}U_r\leq\d^*_1.
\end{align*}
By lemma \ref{lem.key-non-iter}, we obtain
\[
G_\th[u^r]\leq\frac{1}{4}G_1[u^r]+\frac{1}{4}F_1[f^r]^{1/2}+U_1[u^r]^{1/4}.
\]
Rescaling back gives \eqref{eqn.G-th-r}.

Iterating \eqref{eqn.G-th-r} $k$ times for $\th$ yields \eqref{eqn.G-th-k}.
Choosing $r=1$ in \eqref{eqn.G-th-r} gives the case $k=1$.
For $k=2$, taking $r=\th$ in \eqref{eqn.G-th-r} yields
\[
G_{\th^2}
\leq\frac{1}{4}G_\th+\frac{1}{4}F_\th^{1/2}+U_\th^{1/4}
\leq\frac{1}{4^2}G_1+\frac{2}{4^2}F_1^{1/2}+\sum_{i=0}^{1}\frac{1}{4^{1-i}}U_{\th^i}^{1/4}.
\]
where we used Lemma \ref{lem.Frk} in the second inequality, since $\th<\eta_p$.

Now, assume that \eqref{eqn.G-th-k} holds for some $k\in\Z_{\geq2}$.
We show it also holds for $k+1$:
\begin{align*}
G_{\th^{k+1}}
&\leq\frac{1}{4}G_{\th^k}+\frac{1}{4}F_{\th^k}^{1/2}+U_{\th^k}^{1/4}
\\
&\leq\frac{1}{4^{k+1}}G_1
+\frac{k}{4^{k+1}}F_1^{1/2}
+\sum_{i=0}^{k-1}\frac{1}{4^{k-i}}U_{\th^i}^{1/4}
+\frac{1}{4^{k+1}}F_1^{1/2}+U_{\th^k}^{1/4}
\\
&=\frac{1}{4^{k+1}}G_1+\frac{k+1}{4^{k+1}}F_1^{1/2}
+\sum_{i=0}^{(k+1)-1}\frac{1}{4^{(k+1)-1-i}}U_{\th^i}^{1/4}.
\end{align*}
This completes the proof.
\end{proof}

\begin{proposition}\label{pro.U-delta}
There exist universal constants $\d_1=\d_(p)>0$ and $0<\a=\a(p)<1$
such that if $u$ is a suitable weak solution of \eqref{eqn.SGM-f}
with $f \in L^p(Q_1)$ satisfying
\[
\sup_{0<r\leq1}U_r\leq\d_1,
\]
then $u\in C^{\a,\a/4}(Q_{r_0/4})$ for some $r_0>0$.
\end{proposition}

\begin{proof}
Let $\d_0$ and $\d^*_1$ be the constants from Theorem \ref{thm.main-G+F}
and Lemma \ref{lem.G-th^k}, respectively.
Set
\[
\d_1:=\min\set{\d_0^4/8^4,\d^*_1}.
\]
Since $\sup_{0<r\leq1}U_r\leq\d_1\leq\d_1^*$,
it follows from \eqref{eqn.G-th-k} that
\[
G_{\th^k}\leq\frac{1}{4^k}G_1
+\frac{k}{4^k}F_1^{1/2}
+\sum_{i=0}^{k-1}\frac{1}{4^{k-1-i}}U_{\th^i}^{1/4}.
\]
Hence
\begin{align*}
G_{\th^k}+F_{\th^k}^{1/2}
&\leq\frac{1}{4^k}G_1+\frac{k}{4^k}F_1^{1/2}
+\sum_{i=0}^{k-1}\frac{1}{4^{k-1-i}}U_{\th^i}^{1/4}
+\frac{1}{4^k}F_1^{1/2}
\\
&\leq\frac{1}{4^k}G_1
+\frac{k+1}{4^k}F_1^{1/2}+\frac{\d_0}{8}\frb{1+\frac{1}{4}+\cdots+\frac{1}{4^{k-1}}}
\\
&\leq \frac{1}{2^k}G_1+\frac{1}{2^k}F_1^{1/2}+\frac{\d_0}{4},
\end{align*}
where we used Lemma \ref{lem.Frk} in the first inequality,
and $\sup_{0<r\leq1}U_r\leq\d_1\leq\d_0^4/8^4$ in the second inequality.

Therefore, for sufficiently large
$k=k_0:=\big\lceil\log_2\frb{(G_1+F_1^{1/2})/\d_0}\big\rceil+3$,
it holds that
\[
G_{\th^{k_0}}+F_{\th^{k_0}}^{1/2}
\leq\frac{\d_0}{8}+\frac{\d_0}{8}+\frac{\d_0}{4}=\frac{\d_0}{2}.
\]
By Theorem \ref{thm.main-G+F},
we have $u\in C^{\a,\a/4}(r_0/4)$,
where $r_0:=\th^{k_0}\leq\th^3\frb{(G_1+F_1^{1/2})/\d_0}^{\ln\th/\ln2}$.
\end{proof}

Next, we present the proof of Theorem \ref{thm.main-U}.

\begin{proof}[Proof of Theorem \ref{thm.main-U}]
Let $\d_1$ be the constant from Proposition~\ref{pro.U-delta}.
Note that $0<r/\rho\leq1$.
Consider the rescaled function $u^\rho(x,t)=u(\rho x,\rho^4t)$, which satisfies
\[
U_\frac{r}{\rho}[u^\rho]=U_r[u]\leq\sup_{0<r\leq\rho}U_r\leq\d_1.
\]
By Proposition~\ref{pro.U-delta},
we have $u^\rho\in C^{\a,\a/4}(Q_{r_0/4})$,
and hence $u\in C^{\a,\a/4}(Q_{\rho r_0/4})$.
\end{proof}

\subsection{Regularity criterion in terms of $L_r$}

This subsection is devoted to proving the regularity criterion in terms of the quantity $L_r$.
We emphasize that this criterion is formulated on the domain
$Q_1\subset\mathbb T\times\R$,
since the following interpolation inequality is required:
\begin{equation}\label{eqn.G-O-L}
G_r \ls O_r^{5/24}L_r^{7/24} \quad \text{for any}\ \, 0<r\leq1.
\end{equation}
The above inequality was proved in \cite[Lemma 5.1]{OR19} for the spatially periodic domain,
and its proof relies essentially on the Fourier expansion of $u$.
This is precisely why the periodic spatial setting is necessary for our argument.

Following the same argument as in the proof of \eqref{eqn.U-L-U}
but replacing $u$ by $u-u_{B_r}$, we only obtain
\[
G_r\ls O_r^{5/24}L_r^{7/24}+O_r^{1/2}.
\]
However,
this weaker inequality is insufficient for our subsequent analysis.

\begin{lemma}
Let $\eta_p$ be the constant from Lemma~\ref{lem.Frk}.
There exist universal constants $0<\d^*_2=\d_2^*(p)<1$ and $0<\th=\th(p)<\eta_p<1$
such that if $u$ is a suitable weak solution of \eqref{eqn.SGM-f}
with $f\in L^p(Q_1)$ satisfying $L_1<\d^*_2$,
then
\[
G_{\th}\leq\frac{1}{4}G_1+\frac{1}{4}F_1^{1/2}+L_1^{1/4}.
\]
\end{lemma}

\begin{proof}
For any $0<r\leq1/4$, we have
\[
O_r
=\sup_{(-r^4,r^4)}\frac{1}{r}\int_{B_r}|u-u_{B_r}|^2
\leq\sup_{(-r^4,r^4)}\frac{1}{r}\int_{B_r}|u-u_{B_{1/4}}|^2
\ls\frac{1}{r}O_{1/4},
\]
where we used Lemma \ref{lem.mean-int} in the first inequality.

Since $u-u_{Q_{1/2}}$ remains a suitable weak solution of \eqref{eqn.SGM-f},
the local energy estimate yields
\begin{align*}
O_{1/4}
&=\sup_{(-1/4^4,1/4^4)}\int_{B_{1/4}}|u-u_{B_{1/4}}|^2
\ls\sup_{(-1/4^4,1/4^4)}\int_{B_{1/4}}|u-u_{Q_{1/2}}|^2
\\
&\ls\int_{Q_{1/2}}|u-u_{Q_{1/2}}|^2+\int_{Q_{1/2}}|u_x|^2
+\int_{Q_{1/2}}|u_x|^3+\int_{Q_{1/2}}u_x^2|u-u_{Q_{1/2}}|+\int_{Q_{1/2}}f|u-u_{Q_{1/2}}|
\\
&\ls\frb{\int_{Q_{1/2}}|u-u_{Q_{1/2}}|^3}^{2/3}+G_1^2+G_1^3
+\frb{\int_{Q_{1/2}}|u_x|^3}^{2/3}\frb{\int_{Q_{1/2}}|u-u_{Q_{1/2}}|^3}^{1/3}
\\
&\quad\,+\frb{\int_{Q_{1/2}}|f|^{3/2}}^{2/3}\frb{\int_{Q_{1/2}}|u-u_{Q_{1/2}}|^3}^{1/3}
\\
&\ls_p(G_1^3+G_1^6+F_1^3)^{2/3}+G_1^2+G_1^3+G_1^2\frb{G_1^3+G_1^6+F_1^3}^{1/3}
+F_1\frb{G_1^3+G_1^6+F_1^3}^{1/3}
\\
&\ls_p G_1^2+G_1^3+G_1^4+F_1^2+G_1^2F_1+G_1F_1\ls G_1^4+F_1^2+G_1^2F_1.
\end{align*}
where we used Lemma \ref{lem.mean-int} in the first inequality,
the fourth inequality follows from the parabolic Poincar\'e inequality,
and, for simplicity,
only the highest powers of $G_1$ and $F_1$ are retained in the last step.

Similarly,
\[
L_r=\frac{1}{r}\int_{Q_r}u_{xx}^2\leq\frac{1}{r}\int_{Q_1}u_{xx}^2=\frac{1}{r}L_1.
\]
Combining these estimates with \eqref{eqn.G-O-L} yields
\begin{align*}
G_r
&\ls O_r^{5/24}L_r^{7/24}\ls r^{-1/2}O_{1/4}^{5/24}L_1^{7/24}
\ls_pr^{-1/2}(G_1^4+F_1^2+G_1^2F_1)^{5/24}L_1^{7/24}
\\
&\ls_p r^{1/2}G_1^{5/6}r^{-1}L_1^{7/24}
+r^{1/2}F_1^{5/12}r^{-1}L_1^{7/24}+r^{-1/2}G_1^{5/12}F_1^{5/24}L_1^{7/24}
\\
&\ls_p r^{3/5}G_1+r^{-6}L_1^{7/4}
+r^{6/5}F_1^{1/2}+r^{-6}L_1^{7/4}
+r^{1/2}G_1^{5/6}r^{-1}L_1^{1/4}+r^{1/2}F_1^{5/12}r^{-1}L_1^{1/3}
\\
&\ls_p r^{3/5}G_1+r^{-6}L_1^{7/4}
+r^{6/5}F_1^{1/2}
+r^{3/5}G_1+r^{-6}L_1^{3/2}+r^{3/5}F_1^{1/2}+r^{-6}L_1^2
\\
&\ls_p r^{3/5}G_1+(r^{3/5}+r^{6/5})F_1^{1/2}
+(r^{-6}L_1^{7/4}+r^{-6}L_1^{3/2}+r^{-6}L_1^2)
\\
&\leq Cr^{3/5}G_1+C(r^{3/5}+r^{6/5})F_1^{1/2}
+C(r^{-6}L_1^{3/2}+r^{-6}L_1^{5/4}+r^{-6}L_1^{7/4})L_1^{1/4}.
\end{align*}
Take $r=\th<\min\set{1/4,\eta_p}$ sufficiently small
such that $C(\th^{3/5}+\th^{6/5})<1/4$,
and then choose $0<\d_2^*<1$ small enough to ensure
\[
C(\th^{-6}L_1^{3/2}+\th^{-6}L_1^{5/4}+\th^{-6}L_1^{7/4})<
C(\th^{-6}(\d^*_2)^{3/2}+\th^{-6}(\d^*_2)^{5/4}+\th^{-6}(\d^*_2)^{7/4})<1.
\]
This completes the proof.
\end{proof}

The following result can be derived by the same argument as in Proposition \ref{pro.U-delta}.

\begin{proposition}
There exist universal constants $\d_2:=\min\set{\d_0^4/8^4,\d^*_2}>0$ and $0<\a<1$
such that if $u$ is a suitable weak solution
of \eqref{eqn.SGM-f} with $f\in L^p(Q_1)$ satisfying
\[
\sup_{0<r\leq1}L_r\leq\d_2,
\]
then $u\in C^{\a,\a/4}(Q_{r_0/4})$ for some $r_0>0$.
\end{proposition}

The proof of Theorem \ref{thm.main-L}
is analogous to the preceding case and is therefore omitted.

\section*{Acknowledgments}

This work was  partially supported by the
National Natural Science Foundation of China (Nos.~12171389 and 11801015).

\section*{Conflict of Interest}
The authors declare that there is no conflict of interest.

\section*{Data Availability}
No data were generated or analyzed in this study.

\renewcommand\refname{References}

\end{document}